\documentclass[a4paper,12pt]{amsart}

\usepackage{amssymb}
\usepackage{amsmath}
\usepackage{amsthm}

\usepackage[colorlinks]{hyperref}

\allowdisplaybreaks \numberwithin{equation}{section}

\theoremstyle{plain}
\newtheorem{theorem}{Theorem}[section]

\newtheorem{assumption}[theorem]{Assumption}

\theoremstyle{definition}
\newtheorem{defi}[theorem]{Definition}

\newtheorem{problem}[theorem]{Problem}
\def\L{{\mathcal L}}

\def\H{\mathcal H}

\begin{document}

\title[Time-fractional pseudo-parabolic equation]{Direct and Inverse problems for time-fractional pseudo-parabolic equations}

\author[M. Ruzhansky]{Michael Ruzhansky}
\address{
  Michael Ruzhansky:
  \endgraf
Department of Mathematics: Analysis,
Logic and Discrete Mathematics
  \endgraf
Ghent University, Belgium
  \endgraf
 and
  \endgraf
 School of Mathematical Sciences
 \endgraf
Queen Mary University of London
\endgraf
United Kingdom
\endgraf
  {\it E-mail address} {\rm michael.ruzhansky@ugent.be}
 }

\author[D. Serikbaev]{Daurenbek Serikbaev}
\address{
  Daurenbek Serikbaev:
   \endgraf
  Department of Mathematics: Analysis, Logic and Discrete Mathematics
  \endgraf
  Ghent University, Belgium
  \endgraf
  and
  \endgraf
   Al--Farabi Kazakh National University
  \endgraf
  Almaty, Kazakhstan
  \endgraf
    and
  \endgraf
   Institute of Mathematics and Mathematical Modeling
  \endgraf
  Almaty, Kazakhstan
  \endgraf
  {\it E-mail address} {\rm daurenbek.serikbaev@ugent.be}
  }

\author[N. Tokmagambetov]{Niyaz Tokmagambetov}
\address{
  Niyaz Tokmagambetov:
   \endgraf
  Department of Mathematics: Analysis, Logic and Discrete Mathematics
  \endgraf
  Ghent University, Belgium
  \endgraf
  and
  \endgraf
   Al--Farabi Kazakh National University
  \endgraf
  Almaty, Kazakhstan
  \endgraf
  {\it E-mail address} {\rm tokmagambetov@math.kz}
 }
\author[B. T. Torebek]{Berikbol T. Torebek}
\address{
  Berikbol T. Torebek:
   \endgraf
  Department of Mathematics: Analysis, Logic and Discrete Mathematics
  \endgraf
  Ghent University, Belgium
  \endgraf
  and
  \endgraf
   Al--Farabi Kazakh National University
  \endgraf
  Almaty, Kazakhstan
  \endgraf
    and
  \endgraf
   Institute of Mathematics and Mathematical Modeling
  \endgraf
  Almaty, Kazakhstan
  \endgraf
  {\it E-mail address} {\rm berikbol.torebek@ugent.be}
  }

\date{\today}

\thanks{The authors were supported by the FWO Odysseus 1 grant G.0H94.18N: Analysis and Partial Differential Equations. MR was supported in parts by the EPSRC Grant
EP/R003025/1, by the Leverhulme Research Grant RPG-2017-151.}


\maketitle

\begin{abstract}
The purpose of this paper is to establish the solvability results to direct and inverse problems for time-fractional pseudo-parabolic equations with the self-adjoint operators. We are especially interested in proving existence and uniqueness of the solutions in the abstract setting of Hilbert spaces.
\end{abstract}
\tableofcontents

\section{Introduction}

The problems of determination of temperature at interior points of a region when the initial and boundary conditions along with diffusion source term are specified are known as direct diffusion conduction problems. In many physical problems, determination of coefficients or the right-hand side (the source term, in case of the diffusion equation) in a differential equation from some available information is required; these problems are known as inverse problems.

Inverse source problems for the diffusion, sub-diffusion and for other types of equations are well studied. There are numerous works published only in recent years in this area (for example, see \cite{Kirane, Lesnic, Yam19, Run18, Slod}). However, inverse problems for pseudo-parabolic equations and for their fractional analogues have been studied relatively little (see \cite{Janno, LT11, LT11b, LV19, Run80, RTT18}).

The inverse problem of determining the coefficient and the
right hand side of a pseudo-parabolic equation from local over defined
states has important applications in various fields of applied science and engineering.
The study of inverse problems for pseudo-parabolic equations began in the 1980s by Rundell (see \cite{Run80}).

Let $\mathcal{H}$ be a separable Hilbert space and let $\mathcal{L},\,\mathcal{M}$ be operators with the corresponding discrete spectra $\{\lambda_\xi\}_{\xi\in\mathcal{I}},\;\{\mu_\xi\}_{\xi\in\mathcal{I}}$ on $\mathcal{H}$, where $\mathcal{I}$ is a countable
set.

In this paper we consider solvability of an inverse source problem for the following pseudo-parabolic equation
\begin{equation}\label{EQ:PP}
\mathcal{D}_t^\alpha [u(t)+\mathcal{L}u(t)]+\mathcal{M}u(t)=f(t) \,\, \hbox{in} \,\, \H,
\end{equation}
for $0<t<T<\infty$, with initial data
\begin{equation}\label{C:IN}
    u(0)=\varphi\in\H,
\end{equation}
and final condition
\begin{equation}\label{C:FIN}
    u(T)=\psi\in\H.
\end{equation}
Here $\mathcal{D}_t^\alpha$ is the Caputo fractional derivative of order $0 <\alpha\leq1$.

In the particular case $\alpha=1,$ the equation \eqref{EQ:PP} coincides with the classical pseudo-parabolic equation with some differential operators $\mathcal{L}$ and $\mathcal{M}.$ The energy for the isotropic material  can be modeled by a pseudo-parabolic equation \cite{CG68}. Some wave processes \cite{BBM72}, filtration of the two-phase flow in porous media with the dynamic capillary pressure \cite{Baren} are also modeled by pseudo-parabolic equations. Time-fractional pseudo-parabolic equation \eqref{EQ:PP} occurs in the study of flows of the Oldroyd-B fluid, one of the most important classes for dilute solutions of polymers \cite{Fet, Tong}.

In this paper, we consider direct and inverse problems for the time-fractional pseudo-parabolic equation with different abstract operators. We seek generalized solutions to these problems in a form of series expansions using the elements of nonharmonic analysis (see \cite{RT16, RTT18}) and we also examine the convergence of the obtained series solutions. The main results on well-posedness of direct and inverse problems are formulated in three theorems.

We will be making the following assumption:
\begin{assumption}\label{A1}
We assume that {\it the operators $\mathcal{L}$ and $\mathcal{M}$ are diagonalisable (can be written in the infinite dimensional matrix form) with respect to some basis $\{e_\xi\}_{\xi\in\mathcal{I}}$ of the separable Hilbert space $\H$ with the eigenvalues $\lambda_\xi\in\mathbb R: \lambda_\xi\geq c_{L} > 0$ and $\mu_\xi\in\mathbb R: \mu_\xi\geq c_{M} > 0$ for all $\xi\in\mathcal{I}$, respectively. Here $c_{L}$ and $c_{M}$ are some constants, $\mathcal{I}$ is some countable set.}

\end{assumption}

We will be sometimes also making the following assumption with $\mathcal{I}=\mathbb N^k$ or $\mathcal{I}=\mathbb Z^k$ for some $k$:

\begin{assumption}\label{A2}
In further calculus for our analysis we will also require that $\lambda_\xi\to\infty$ and $\mu_\xi\to\infty$ as $|\xi|\to\infty$.
Moreover, we will assume that $|\lambda_\xi|=O(|\mu_\xi|^\kappa)$ as $|\xi|\to\infty$ for some $\kappa>0$.
\end{assumption}

\subsection{Preliminaries}

Now, for the formulation of problems we need to define fractional differentiation operators.

\begin{defi}\label{defi: riemann int}
The Riemann-Liouville fractional integral $I^\alpha$ of order $0<\alpha<1$ for an integrable function $f$ is defined by
$$
I^\alpha[f](t)=\frac{1}{\Gamma(\alpha)}\int_c^t{(t-s)^{\alpha-1}f(s)ds},\; t\in[c,d],
$$
where $\Gamma$ denotes the Euler gamma function.

The Caputo fractional derivative of order $0<\alpha<1$ of a differentiable function $f$ is defined by
$$
\mathcal{D}_t^\alpha[f](t)=I^{1-\alpha}[f'(t)]=\frac{1}{\Gamma(1-\alpha)}\int_c^t\frac{f'(s)}{(t-s)^\alpha}ds,\; t\in[c,d].
$$
\end{defi}
For more information see \cite{KST06}.

In what follows, we will widely use the properties of the
Mittag-Leffler type function (see \cite{LG99})
$$
E_{\alpha,\beta}(z)=\sum_{m=0}^{\infty}\frac{z^m}{\Gamma(\alpha m+\beta)}.
$$
In \cite{Sim14} the following estimate for the Mittag-Leffler function is proved, when $0<\alpha<1$ (not true for $\alpha\geq 1$)
\begin{equation}\label{EST: Mittag AB}
\frac{1}{1+\Gamma(1-\alpha)z}\leq E_{\alpha,1}(-z)\leq\frac{1}{1+\Gamma(1+\alpha)^{-1}z},\;z>0.
\end{equation}
Thus, it follows that
\begin{equation}\label{EST: Mittag}
0<E_{\alpha,1}(-z)<1,\;z>0.
\end{equation}

\section{Direct problem} We consider the pseudo-parabolic equation
\begin{equation}\label{EQ DL:Pseudo}
  \mathcal{D}_t^\alpha[u(t)+\mathcal{L}u(t)]+\mathcal{M}u(t)=f(t),
\end{equation}
for $0<t<T<\infty$, with the Cauchy condition
\begin{equation}\label{CON DL:In}
u(0)=\varphi\in\H.
\end{equation}

\begin{defi}
Let $X$ be a separable Hilbert space.
\begin{itemize}
    \item The space $C^\alpha([0,T];X),\; 0<\alpha\leq 1$ is the space of all continuous functions $g:[0,T]\rightarrow X$ with also continuous $\mathcal{D}^\alpha_t g:[0,T]\rightarrow X$, such that
$$
\|g\|_{C^\alpha([0,T]; X)}=\|g\|_{C([0,T]; X)}+\|\mathcal{D}^\alpha_t g\|_{C([0,T]; X)} < \infty.
$$
    \item The space $W^\alpha([0,T];X),\; 0<\alpha\leq 1$ is the space of all $L^2$-integrable functions $g:[0,T]\rightarrow X$ with $L^2$-integrable $\mathcal{D}^\alpha_t g:[0,T]\rightarrow X$, such that
$$
\|g\|_{W^\alpha([0,T]; X)}=\|g\|_{L^2([0,T]; X)}+\|\mathcal{D}^\alpha_t g\|_{L^2([0,T]; X)} < \infty.
$$
\end{itemize}
\end{defi}

A generalised solution of the direct problem \eqref{EQ DL:Pseudo}-\eqref{CON DL:In} is the function $u\in L^2([0,T]; \mathcal{H}^{1}_\mathcal{L}) \cap L^2([0,T]; \mathcal{H}^{1}_\mathcal{M}) \cap W^{\alpha}([0,T]; \mathcal{H}^{1}_\mathcal{L})$. Here we define $\mathcal{H}^{l, m}_{\L,\mathcal{M}}$ as
\begin{equation}
\label{Spaces}
\mathcal{H}^{l, m}_{\L,\mathcal{M}}:=\{u\in\mathcal{H}:\,\mathcal{L}^{l}\mathcal{M}^{m}u\in\mathcal{H}\}, \end{equation}
for any $l, m \in \mathbb R$. In view of this we can define $\H_\L^l,\,\H_\mathcal{M}^m$ correspondingly
$$
\mathcal{H}^{l}_{\L}:=\{u\in\mathcal{H}:\,\mathcal{L}^{l}u\in\mathcal{H}\},
$$
$$
\mathcal{H}^{m}_{\mathcal{M}}:=\{u\in\mathcal{H}:\,\mathcal{M}^{m}u\in\mathcal{H}\},
$$
for any $l, m \in \mathbb R$.

\subsection{Case I: $1/2<\alpha<1$} 

For Problem \eqref{EQ DL:Pseudo}-\eqref{CON DL:In}, the following theorem holds true.

\begin{theorem}\label{Theorem D-L}
Let $1/2<\alpha<1.$ Suppose that Assumption \ref{A1} holds. Let $\varphi\in \mathcal{H}^{1}_\mathcal{L}\cap\mathcal{H}^{1}_\mathcal{M}$ and $f\in L^2([0,T];\mathcal{H})\cap L^{2}([0,T],\mathcal{H}_{\mathcal L, \mathcal M}^{-1, 1})$. Then there exists a unique solution $u(t)$ of Problem \eqref{EQ DL:Pseudo}-\eqref{CON DL:In} such that $u\in  L^2([0,T];\mathcal{H}^{1}_\mathcal{M}) \cap W^{\alpha}([0,T]; \mathcal{H}^{1}_\mathcal{L})$. This solution can be written in the form
\begin{equation*}
\begin{split}
u(t)&=\sum_{\xi\in\mathcal{I}}\varphi_\xi E_{\alpha,1}\left(-\frac{\mu_\xi}{1+\lambda_\xi}t^\alpha\right)e_\xi\\
&+\sum_{\xi\in\mathcal{I}}\left[ \frac{1}{1+\lambda_\xi} \int_0^t s^{\alpha-1} E_{\alpha,\alpha}\left(-\frac{\mu_\xi}{1+\lambda_\xi}s^\alpha\right) f_\xi(t-s) ds\right]e_\xi,
\end{split}
\end{equation*}
where $\varphi_\xi=(\varphi,e_\xi)_\mathcal{H}$, $f_\xi(t)=(f(t),e_\xi)_\mathcal{H}$.
\end{theorem}

\begin{proof} Let us first prove the existence result. Since the system of eigenfunctions $e_\xi$ is a basis in $\mathcal{H}$, we seek the function $u(t)$ in the form
\begin{equation}\label{EQ:D expansion u}
u(t)=\sum_{\xi\in\mathcal{I}}u_\xi(t)e_\xi,
\end{equation}
where $u_\xi(t)$ are unknown functions. Substituting \eqref{EQ:D expansion u} into Equations \eqref{EQ DL:Pseudo}-\eqref{CON DL:In} and taking into account
Assumption \ref{A1},
we obtain the following equations corresponding to the function $u_\xi(t)$:
\begin{equation}\label{EQ DL:ODE}
\mathcal{D}_t^\alpha u_\xi(t)+\frac{\mu_\xi}{1+\lambda_\xi}u_\xi(t)=\frac{f_\xi(t)}{1+\lambda_\xi},
\end{equation}
\begin{equation}\label{CON DL:ODE}
u_\xi(0)=\varphi_{\xi},\;\xi\in\mathcal{I}.
\end{equation}
Here $f_\xi(t)$ is the coefficient function of the expansion of $f(t)$, i.e.
\begin{equation*}\label{EQ:D expansion f}
f(t)=\sum_{\xi\in\mathcal{I}}f_\xi(t)e_\xi,
\end{equation*}
with
$$
f_\xi(t)=(f(t),e_\xi)_\mathcal{H},
$$
and $\varphi_\xi$ is the coefficient of the expansion of $\varphi$, i.e.
\begin{equation*}\label{EQ:D expansion phi}
\varphi=\sum_{\xi\in\mathcal{I}}\varphi_\xi e_\xi,
\end{equation*}
with
\begin{equation}\label{coef phi}
\varphi_\xi=(\varphi,e_\xi)_\mathcal{H}.
\end{equation}

According to \cite{LG99}, the solutions of the equation \eqref{EQ DL:ODE} satisfying the initial condition \eqref{CON DL:ODE} can be represented in the form
\begin{equation}\label{SOL DL:ODE}
\begin{split}
    u_\xi(t)&=\varphi_\xi E_{\alpha,1}\left(-\frac{\mu_\xi}{1+\lambda_\xi}t^\alpha\right)
    -\frac{1}{\mu_\xi}\int_0^t\frac{d}{ds}\left(E_{\alpha,1}\left(-\frac{\mu_\xi}{1+\lambda_\xi}s^\alpha\right)\right)f_\xi(t-s)ds.
\end{split}\end{equation}

Putting \eqref{SOL DL:ODE} into \eqref{EQ:D expansion u}, we obtain a solution of Problem \eqref{EQ DL:Pseudo}-\eqref{CON DL:In}, i.e.
\begin{equation}\label{FIN SOL DL:Pseudo}
\begin{split}
    u(t)&=\sum_{\xi\in\mathcal{I}}\varphi_\xi E_{\alpha,1}\left(-\frac{\mu_\xi}{1+\lambda_\xi}t^\alpha\right)e_\xi\\
    &-\sum_{\xi\in\mathcal{I}}\left[\frac{1}{\mu_\xi}\int_0^t\frac{d}{ds}
    \left(E_{\alpha,1}\left(-\frac{\mu_\xi}{1+\lambda_\xi}s^\alpha\right)\right)f_\xi(t-s)ds\right]e_\xi.
\end{split}
\end{equation}

Here under the  integral we have the derivative of the Mittag-Leffler function, which is a non-positive function, i.e.
\begin{equation*}
\label{ML:non positive}
\begin{split}
\frac{d}{ds}\left(E_{\alpha,1}\left(-\frac{\mu_\xi}{1+\lambda_\xi}s^\alpha\right)\right)= -\frac{\mu_\xi}{1+\lambda_\xi}s^{\alpha-1} E_{\alpha,\alpha}\left(-\frac{\mu_\xi}{1+\lambda_\xi}s^\alpha\right)\leq 0,
\end{split}
\end{equation*}
in view of 
the fact that
\begin{equation*}
\begin{split}
  E_{\alpha,1}^\prime(z)&=\frac{d}{dz}\left(\sum_{k=0}^\infty\frac{z^k}{\Gamma(\alpha k+1)}\right)=\sum_{k=1}^\infty\frac{k z^{k-1}}{\Gamma(\alpha k+1)}\\
  &=\frac{1}{\alpha}\sum_{m=0}^\infty\frac{z^m}{\Gamma(\alpha m+\alpha)}=\frac{1}{\alpha}E_{\alpha,\alpha}(z),\;z\in\mathbb{R}.
\end{split}
\end{equation*}
From this we have
\begin{equation}
\label{Eq:007}
\begin{split}
u(t)&=\sum_{\xi\in\mathcal{I}}\varphi_\xi E_{\alpha,1}\left(-\frac{\mu_\xi}{1+\lambda_\xi}t^\alpha\right)e_\xi\\
&+\sum_{\xi\in\mathcal{I}}\left[\int_0^t s^{\alpha-1} E_{\alpha,\alpha}\left(-\frac{\mu_\xi}{1+\lambda_\xi}s^\alpha\right)  \frac{f_\xi(t-s)}{1+\lambda_\xi} ds\right]e_\xi.
\end{split}
\end{equation}

Now, we prove the convergence of the obtained
infinite series corresponding to the functions $u(t),\;\mathcal{D}_t^\alpha u(t),\;\mathcal{M}u(t)$ and $\mathcal{D}_t^\alpha \mathcal{L}u(t)$.

Before we get the convergence, let us calculate $\mathcal{M}u(t),\;\mathcal{D}_t^\alpha u(t)$ and $\mathcal{D}_t^\alpha \mathcal{L}u(t)$. By using Assumption \ref{A1} in \eqref{coef phi}, we have
\begin{equation}
\label{rem}
\begin{split}
\lambda_\xi\varphi_\xi&=\lambda_\xi(\varphi,e_\xi)_\mathcal{H}
=(\varphi,\mathcal{L}e_\xi)_\mathcal{H}=(\mathcal{L}\varphi,e_\xi)_\mathcal{H};\\
\mu_\xi\varphi_\xi&=\mu_\xi(\varphi,e_\xi)_\mathcal{H}
=(\varphi,\mathcal{M}e_\xi)_\mathcal{H}=(\mathcal{M}\varphi,e_\xi)_\mathcal{H}.
\end{split}
\end{equation}
Applying the operator $\mathcal{L}$ to \eqref{FIN SOL DL:Pseudo}, and taking into account formulas \eqref{rem}, we get
\begin{equation}
\label{EQ DL:Lu}
\begin{split}
\mathcal{L}u(t)&=\sum_{\xi\in\mathcal{I}}\varphi_\xi E_{\alpha,1}\left(-\frac{\mu_\xi}{1+\lambda_\xi}t^\alpha\right)\mathcal{L}e_\xi
\\
&+\sum_{\xi\in\mathcal{I}}\left[\int_0^t s^{\alpha-1} E_{\alpha,\alpha}\left(-\frac{\mu_\xi}{1+\lambda_\xi}s^\alpha\right)  \frac{f_\xi(t-s)}{1+\lambda_\xi} ds\right]\mathcal{L}e_\xi
\\
&=\sum_{\xi\in\mathcal{I}}\lambda_\xi\varphi_\xi E_{\alpha,1}\left(-\frac{\mu_\xi}{1+\lambda_\xi}t^\alpha\right)e_\xi
\\
&+\sum_{\xi\in\mathcal{I}}\frac{\lambda_\xi}{1+\lambda_\xi}\left[\int_0^t s^{\alpha-1} E_{\alpha,\alpha}\left(-\frac{\mu_\xi}{1+\lambda_\xi}s^\alpha\right) f_\xi(t-s)  ds\right]e_\xi
\\
&=\sum_{\xi\in\mathcal{I}}(\mathcal{L}\varphi,e_\xi)_\mathcal{H}E_{\alpha,1}\left(-\frac{\mu_\xi}{1+\lambda_\xi}t^\alpha\right)e_\xi
\\
&+\sum_{\xi\in\mathcal{I}}\frac{\lambda_\xi}{1+\lambda_\xi}\left[\int_0^t s^{\alpha-1} E_{\alpha,\alpha}\left(-\frac{\mu_\xi}{1+\lambda_\xi}s^\alpha\right) f_\xi(t-s)  ds\right]e_\xi.
\end{split}
\end{equation}
Analogously, we have
\begin{equation}
\label{EQ DL:Lu2}
\begin{split}
\mathcal{M}u(t)&=\sum_{\xi\in\mathcal{I}}(\mathcal{M}\varphi,e_\xi)_\mathcal{H}E_{\alpha,1}\left(-\frac{\mu_\xi}{1+\lambda_\xi}t^\alpha\right)e_\xi\\
&+\sum_{\xi\in\mathcal{I}}\frac{\mu_\xi}{1+\lambda_\xi}\left[\int_0^t s^{\alpha-1} E_{\alpha,\alpha}\left(-\frac{\mu_\xi}{1+\lambda_\xi}s^\alpha\right) f_\xi(t-s)  ds\right]e_\xi.
\end{split}
\end{equation}

Applying the operator $\mathcal{D}_t^\alpha$ to \eqref{EQ:D expansion u}, we have
\begin{equation}\label{EQ:D expansion Du}
\mathcal{D}_t^\alpha u(t)=\sum_{\xi\in\mathcal{I}}\mathcal{D}_t^\alpha u_\xi(t)e_\xi.
\end{equation}
By using \eqref{EQ DL:ODE}, we find $\mathcal{D}_t^\alpha u_\xi(t)$
\begin{equation}
\label{EQ DL:Duh}
\begin{split}
   \mathcal{D}_t^\alpha u_\xi(t)=\frac{f_\xi(t)}{1+\lambda_\xi}-\frac{\mu_\xi}{1+\lambda_\xi}u_\xi(t).
\end{split}
\end{equation}
Putting \eqref{Eq:007} into \eqref{EQ DL:Duh}, we get
\begin{equation}\label{EQ DL:Duht}\begin{split}
   \mathcal{D}_t^\alpha u_\xi(t)&=\frac{f_\xi(t)}{1+\lambda_\xi}-\frac{\mu_\xi}{1+\lambda_\xi}\varphi_\xi E_{\alpha,1}\left(-\frac{\mu_\xi}{1+\lambda_\xi}t^\alpha\right)\\
   &+\frac{\mu_\xi}{(1+\lambda_\xi)^2}\left[\int_0^t s^{\alpha-1} E_{\alpha,\alpha}\left(-\frac{\mu_\xi}{1+\lambda_\xi}s^\alpha\right) f_\xi(t-s)  ds\right].
\end{split}\end{equation}
Substituting \eqref{EQ DL:Duht} into \eqref{EQ:D expansion Du}, we obtain
\begin{equation}
\label{EQ DL:Du}
\begin{split}
\mathcal{D}_t^\alpha u(t)&=\sum_{\xi\in\mathcal{I}}\frac{f_\xi(t)}{1+\lambda_\xi}e_\xi
-\sum_{\xi\in\mathcal{I}}\frac{\mu_\xi}{1+\lambda_\xi}\varphi_\xi E_{\alpha,1}\left(-\frac{\mu_\xi}{1+\lambda_\xi}t^\alpha\right)e_\xi\\
&+\sum_{\xi\in\mathcal{I}}\frac{\mu_\xi}{(1+\lambda_\xi)^2}\left[\int_0^t s^{\alpha-1} E_{\alpha,\alpha}\left(-\frac{\mu_\xi}{1+\lambda_\xi}s^\alpha\right) f_\xi(t-s)  ds\right]e_\xi.
\end{split}
\end{equation}
Applying the operator $\mathcal{L}$ to \eqref{EQ DL:Du} and taking into account formulas \eqref{rem}, we have
\begin{equation}
\label{EQ DL:DLu}
\begin{split}
\mathcal{D}_t^\alpha\mathcal{L} u(t)&=\sum_{\xi\in\mathcal{I}}\frac{\lambda_\xi }{1+\lambda_\xi} f_\xi(t) e_\xi-\sum_{\xi\in\mathcal{I}}\frac{\mu_\xi}{1+\lambda_\xi}(\mathcal{L}\varphi, e_\xi)_\mathcal{H} E_{\alpha,1}\left(-\frac{\mu_\xi}{1+\lambda_\xi}t^\alpha\right)e_\xi\\
&+\sum_{\xi\in\mathcal{I}}\frac{\mu_\xi\lambda_\xi}{(1+\lambda_\xi)^2}\left[\int_0^t s^{\alpha-1} E_{\alpha,\alpha}\left(-\frac{\mu_\xi}{1+\lambda_\xi}s^\alpha\right) f_\xi(t-s)  ds\right]e_\xi e_\xi.
\end{split}
\end{equation}

Now let us estimate $\mathcal{H}$-norms
\begin{equation}
\label{C1: Sol1-mod}
\begin{split}
\|u(t)\|_{\mathcal{H}}^{2} &\leq \sum_{\xi\in\mathcal{I}}\left|E_{\alpha,1}\left(-\frac{\mu_\xi}{1+\lambda_\xi}t^\alpha\right) \right|^{2} |\varphi_\xi|^{2}
\\
& + \sum_{\xi\in\mathcal{I}}\frac{1}{(1+\lambda_\xi)^{2}}\left|\int_0^t s^{\alpha-1}  E_{\alpha,\alpha}\left(-\frac{\mu_\xi}{1+\lambda_\xi}s^\alpha\right)  f_\xi(t-s) ds\right|^{2}
\\
&\leq C \sum\limits_{\xi\in\mathcal{I}} \left(\frac{1}{1+\frac{\mu_\xi}{1+\lambda_\xi}t^\alpha}\right)^{2} |\varphi_{\xi}|^{2} \\ & + C \sum\limits_{\xi\in\mathcal{I}}\frac{1}{(1+\lambda_\xi)^{2}}\left(\int\limits_0^t \frac{s^{\alpha-1}}{1+\frac{\mu_\xi}{1+\lambda_\xi}s^\alpha} |f_{\xi}(t-s)| d s\right)^{2}
\\
&\leq C \sum\limits_{\xi\in\mathcal{I}} |\varphi_{\xi}|^{2}
 + C \sum\limits_{\xi\in\mathcal{I}}\frac{1}{(1+\lambda_\xi)^{2}}\int\limits_0^t \left(\frac{s^{\alpha-1}}{1 + \frac{\mu_\xi}{1+\lambda_\xi} s^\alpha}\right)^{2}  d s \,
\int\limits_0^t |f_{\xi}(s)|^{2} d s
\\
&\leq C\|\varphi\|_\mathcal{H}^2
 + C \sum\limits_{\xi\in\mathcal{I}}\frac{1}{(1+\lambda_\xi)^{2}}\int\limits_0^t \left(\frac{1}{s^{1-\alpha} + \frac{\mu_\xi}{1+\lambda_\xi}s }\right)^{2} d s \,
\int\limits_0^t |f_{\xi}(s)|^{2} d s
\\
&\leq C\|\varphi\|_\mathcal{H}^2
 + C \sum\limits_{\xi\in\mathcal{I}}\frac{1}{(1+\lambda_\xi)^{2}}\int\limits_0^t \frac{1}{s^{2-2\alpha} }  d s \,
\int\limits_0^t |f_{\xi}(s)|^{2} d s
\end{split}
\end{equation}
Due to the assumption $\alpha>1/2$, finally, we get
$$
\|u\|_{L^{2}(0, T; \mathcal H)}^{2}\leq C(\|\varphi\|_\mathcal{H}^2
 + \|f\|^{2}_{L^{2}([0, T]; \mathcal H_{\mathcal L}^{-1})}).
$$
Here we take into account that
\begin{equation}
\label{H}
\begin{split}
\sum\limits_{\xi\in\mathcal{I}}\frac{1}{(1+\lambda_\xi)^{2}}
\int\limits_0^t |f_{\xi}(s)|^{2} d s &= 
\int\limits_0^t \sum\limits_{\xi\in\mathcal{I}} \left|\frac{f_{\xi}(s)}{1+\lambda_\xi}\right|^{2} d s
\\
&=
\sum\limits_{\xi\in\mathcal{I}}
\int\limits_0^t \|f(s)\|^{2}_{\mathcal H_{\mathcal L}^{-1}} d s \leq C \|f\|^{2}_{L^{2}([0, T]; \mathcal H_{\mathcal L}^{-1})},
\end{split}
\end{equation}
for some constant $C>0$.

Finally, by using \eqref{EST: Mittag} and arguing as in \eqref{C1: Sol1-mod}, from \eqref{FIN SOL DL:Pseudo}--\eqref{EQ DL:DLu} we get the following estimates
\begin{equation*}
\begin{split}
    \|\mathcal{L}u\|_{L^{2}([0,T],\mathcal{H})}^2&\leq C(\|\mathcal{L}\varphi\|_\mathcal{H}^2 + \|f\|_{L^{2}([0,T],\mathcal{H})}^2),
\end{split}
\end{equation*}

\begin{equation*}
\begin{split}
    \|\mathcal{M}u\|_{L^{2}([0,T],\mathcal{H})}^2&\leq C(\|\mathcal{M}\varphi\|_\mathcal{H}^2 + \|f\|_{L^{2}([0,T],\mathcal{H}_{\mathcal L, \mathcal M}^{-1, 1})}^2),
\end{split}
\end{equation*}

\begin{equation*}
\begin{split}
  \|\mathcal{D}_t^\alpha u\|_{L^{2}([0,T],\mathcal{H})}^2&\leq C(\|\varphi\|_{\mathcal{H}^{-1, 1}_{\mathcal L, \mathcal M}}^2 +\|f\|_{L^{2}([0,T],\mathcal{H}_{\mathcal L}^{-1})}^2+\|f\|_{L^{2}([0,T],\mathcal{H}_{\mathcal L, \mathcal M}^{-2, 1})}^2),
\end{split}
\end{equation*}
and
\begin{equation*}
\begin{split}
\|\mathcal{D}_t^\alpha\mathcal{L} u\|_{L^{2}([0,T],\mathcal{H})}^2&\leq C(\|\varphi\|_{\mathcal{H}^{1}_{\mathcal M}}^2 +\|f\|_{L^{2}([0,T],\mathcal{H})}^2 +\|f\|_{L^{2}([0,T],\mathcal{H}_{\mathcal L, \mathcal M}^{-1, 1})}^2),
\end{split}
\end{equation*}
respectively. Here in all our estimates in the spaces $\mathcal{H}_{\mathcal L, \mathcal M}^{l, m}$ for some $l, m\in\mathbb R$ we play with the argument as in \eqref{H}. Thus, we finish the proof of the existence result.

{\it Proof of the uniqueness of the solution.} Let $w(t)$ and $v(t)$ be two solutions of Problem \eqref{EQ DL:Pseudo}-\eqref{CON DL:In}, i.e.
$$
\mathcal{D}_t^\alpha w(t)+\mathcal{D}_t^\alpha\mathcal{L}w(t)+\mathcal{M}w(t)=f(t),
$$
$$
w(0)=\varphi,
$$
$$
\mathcal{D}_t^\alpha v(t)+\mathcal{D}_t^\alpha\mathcal{L}v(t)+\mathcal{M}v(t)=f(t),
$$
$$
v(0)=\varphi.
$$
By subtracting these equations from each other, and denoting $u(t)=w(t)-v(t)$, we obtain
\begin{equation}
\label{EQ:HOM P}
\mathcal{D}_t^\alpha u(t)+\mathcal{D}_t^\alpha\mathcal{L}u(t)+\mathcal{M}u(t)=0,
\end{equation}
\begin{equation}
\label{CON:HOM P}
u(0)=0.
\end{equation}
We also have
\begin{equation}
\label{uxi}
u_\xi(t)=(u(t),e_\xi)_\mathcal{H},\;\xi\in\mathcal{I}.
\end{equation}
Applying the operator $\mathcal{D}_t^\alpha$ to \eqref{uxi}, we
have
\begin{equation}
\label{Duxi}
\mathcal{D}_t^\alpha u_\xi(t)=(\mathcal{D}_t^\alpha u(t),e_\xi)_\mathcal{H},\;\xi\in\mathcal{I}.
\end{equation}

From \eqref{EQ:HOM P}--\eqref{CON:HOM P}, we have
\begin{equation}
\label{EQ:HOM ODE}
\mathcal{D}_t^\alpha u_\xi(t)+\frac{\mu_\xi}{1+\lambda_\xi}u_\xi(t)=0,
\end{equation}
\begin{equation}
\label{CON:HOM ODE}
u_\xi(0)=0.
\end{equation}
By the formula \eqref{SOL DL:ODE}, when $\varphi_\xi=0,\;f_\xi(t)=0$, the solution of the problem \eqref{EQ:HOM ODE}--\eqref{CON:HOM ODE} is $u_\xi(t)\equiv 0$.

Further, by the basis property of the system $\{e_\xi\}_{\xi\in\mathcal I}$ in $\mathcal{H}$, we obtain $u(t)\equiv 0$. The uniqueness of the solution of Problem \eqref{EQ DL:Pseudo}--\eqref{CON DL:In} is proved.
\end{proof}

\subsection{Case II: $0<\alpha<1$} 
Here we deal with the case when $0<\alpha<1$. But for this we will require more conditions on source term.

\begin{theorem}
\label{Theorem D-L-2}
Let $0<\alpha<1.$ Suppose that Assumption \ref{A1} holds. Let $\varphi\in \mathcal{H}^{1}_\mathcal{L}\cap\mathcal{H}^{1}_\mathcal{M}$ and $f\in W^{1}([0,T];\mathcal{H})$.
Then there exists a unique solution $u(t)$ of Problem \eqref{EQ DL:Pseudo}-\eqref{CON DL:In} such that $u\in L^2([0,T]; \mathcal{H}^{1}_\mathcal{M}) \cap W^{\alpha}([0,T]; \mathcal{H}^{1}_\mathcal{L})$. This solution can be written in the form
\begin{equation*}\begin{split}
    u(t)&=\sum_{\xi\in\mathcal{I}}\varphi_\xi E_{\alpha,1}\left(-\frac{\mu_\xi}{1+\lambda_\xi}t^\alpha\right)e_\xi
    \\
    &+\sum_{\xi\in\mathcal{I}}\frac{f_\xi(t)}{\mu_\xi} e_\xi
   - \sum_{\xi\in\mathcal{I}}\frac{f_\xi(0)}{\mu_\xi} E_{\alpha,1}\left(-\frac{\mu_\xi}{1+\lambda_\xi}t^\alpha\right) e_\xi
   \\
   & - \sum_{\xi\in\mathcal{I}}\left[\int_0^t \frac{f'_{\xi}(t-s)}{\mu_\xi} E_{\alpha,1}\left(-\frac{\mu_\xi}{1+\lambda_\xi}s^\alpha\right) ds\right] e_\xi,
\end{split}\end{equation*}
where $\varphi_\xi=(\varphi,e_\xi)_\mathcal{H}$, $f_\xi(t)=(f(t),e_\xi)_\mathcal{H}$.
\end{theorem}

\begin{proof} By repeating the arguments of Theorem \ref{Theorem D-L}, we start from the formula \eqref{SOL DL:ODE}. For the last term of the equation \eqref{SOL DL:ODE}, we have
\begin{equation}\label{SOL DL:ODE+}
\begin{split}
     -\frac{1}{\mu_\xi}&\int_0^t\frac{d}{ds} \left(E_{\alpha,1}\left(-\frac{\mu_\xi}{1+\lambda_\xi}s^\alpha\right)\right)f_\xi(t-s)ds
    \\
    = & - \frac{1}{\mu_\xi} E_{\alpha,1}\left(-\frac{\mu_\xi}{1+\lambda_\xi}s^\alpha\right) f_\xi(t-s)\Big|_{0}^{t}
    \\
    & - \frac{1}{\mu_\xi}\int_0^t E_{\alpha,1}\left(-\frac{\mu_\xi}{1+\lambda_\xi}s^\alpha\right) f'_{\xi}(t-s)ds
    \\
    = &\frac{1}{\mu_\xi} f_\xi(t) - \frac{1}{\mu_\xi} E_{\alpha,1}\left(-\frac{\mu_\xi}{1+\lambda_\xi}t^\alpha\right) f_\xi(0)
    \\
    & - \frac{1}{\mu_\xi}\int_0^t E_{\alpha,1}\left(-\frac{\mu_\xi}{1+\lambda_\xi}s^\alpha\right) f'_{\xi}(t-s)ds.
\end{split}
\end{equation}
Thus, for the solution of the Cauchy problem
\begin{equation*}
\mathcal{D}_t^\alpha u_\xi(t)+\frac{\mu_\xi}{1+\lambda_\xi}u_\xi(t)=\frac{f_\xi(t)}{1+\lambda_\xi}, \,\,\,\, u_\xi(0)=\varphi_{\xi},
\end{equation*}
we have
\begin{equation}\label{SOL DL:ODE++}
\begin{split}
    u_\xi(t)&=\varphi_\xi E_{\alpha,1}\left(-\frac{\mu_\xi}{1+\lambda_\xi}t^\alpha\right)+\frac{1}{\mu_\xi} f_\xi(t)
   \\ &- \frac{1}{\mu_\xi} E_{\alpha,1}\left(-\frac{\mu_\xi}{1+\lambda_\xi}t^\alpha\right) f_\xi(0)\\& - \frac{1}{\mu_\xi}\int_0^t E_{\alpha,1}\left(-\frac{\mu_\xi}{1+\lambda_\xi}s^\alpha\right) f'_{\xi}(t-s)ds,
\end{split}
\end{equation}
for all $\xi\in\mathcal{I}.$

Putting \eqref{SOL DL:ODE++} into \eqref{EQ:D expansion u}, we obtain the solution of Problem \eqref{EQ DL:Pseudo}-\eqref{CON DL:In} in the following form
\begin{equation}
\label{FIN SOL DL:Pseudo-2}
\begin{split}
    u(t)&=\sum_{\xi\in\mathcal{I}}\varphi_\xi E_{\alpha,1}\left(-\frac{\mu_\xi}{1+\lambda_\xi}t^\alpha\right)e_\xi
    \\
    &+\sum_{\xi\in\mathcal{I}}\frac{f_\xi(t)}{\mu_\xi} e_\xi
   - \sum_{\xi\in\mathcal{I}}\frac{f_\xi(0)}{\mu_\xi} E_{\alpha,1}\left(-\frac{\mu_\xi}{1+\lambda_\xi}t^\alpha\right) e_\xi
   \\
   & - \sum_{\xi\in\mathcal{I}}\left[\int_0^t \frac{f'_{\xi}(t-s)}{\mu_\xi} E_{\alpha,1}\left(-\frac{\mu_\xi}{1+\lambda_\xi}s^\alpha\right) ds\right] e_\xi.
\end{split}
\end{equation}

To prove the convergence of the obtained
infinite series corresponding to the functions $\mathcal{L} u(t)$, $\mathcal{M} u(t)$, $\mathcal{D}_t^\alpha u(t)$ and $\mathcal{D}_t^\alpha \mathcal{L} u(t)$, first, we need to calculate them.

Applying the operator $\mathcal{L}$ to \eqref{FIN SOL DL:Pseudo}, and taking into account formulas \eqref{rem}, we get
\begin{equation}
\label{EQ DL:Lu-2}
\begin{split}
\mathcal{L}u(t)&
=\sum_{\xi\in\mathcal{I}}(\mathcal{L}\varphi,e_\xi)_\mathcal{H} E_{\alpha,1}\left(-\frac{\mu_\xi}{1+\lambda_\xi}t^\alpha\right)e_\xi
    \\
    &+\sum_{\xi\in\mathcal{I}}\frac{\lambda_\xi }{\mu_\xi} f_\xi(t) e_\xi
   - \sum_{\xi\in\mathcal{I}}\frac{\lambda_\xi }{\mu_\xi} f_\xi(0) E_{\alpha,1}\left(-\frac{\mu_\xi}{1+\lambda_\xi}t^\alpha\right) e_\xi
   \\
   & - \sum_{\xi\in\mathcal{I}}\frac{\lambda_\xi }{\mu_\xi}\left[\int_0^t  f'_{\xi}(t-s) E_{\alpha,1}\left(-\frac{\mu_\xi}{1+\lambda_\xi}s^\alpha\right) ds\right] e_\xi.
\end{split}
\end{equation}
Analogously, we have
\begin{equation}
\label{EQ DL:Lu2-2}
\begin{split}
\mathcal{M}u(t)&=\sum_{\xi\in\mathcal{I}}(\mathcal{M}\varphi,e_\xi)_\mathcal{H}E_{\alpha,1}\left(-\frac{\mu_\xi}{1+\lambda_\xi}t^\alpha\right)e_\xi
   \\
    & + \sum_{\xi\in\mathcal{I}} f_\xi(t) e_\xi
    - \sum_{\xi\in\mathcal{I}} f_\xi(0) E_{\alpha,1}\left(-\frac{\mu_\xi}{1+\lambda_\xi}t^\alpha\right) e_\xi
   \\
   & - \sum_{\xi\in\mathcal{I}} \left[\int_0^t  f'_{\xi}(t-s) E_{\alpha,1}\left(-\frac{\mu_\xi}{1+\lambda_\xi}s^\alpha\right) ds\right] e_\xi.
\end{split}
\end{equation}

Applying the operator $\mathcal{D}_t^\alpha$ to \eqref{EQ:D expansion u}, we have
\begin{equation}
\label{EQ:D expansion Du-2}
\mathcal{D}_t^\alpha u(t)=\sum_{\xi\in\mathcal{I}}\mathcal{D}_t^\alpha u_\xi(t)e_\xi.
\end{equation}
By using \eqref{EQ DL:ODE}, we find $\mathcal{D}_t^\alpha u_\xi(t)$
\begin{equation}
\label{EQ DL:Duh-2}
\begin{split}
\mathcal{D}_t^\alpha u_\xi(t)=\frac{f_\xi(t)}{1+\lambda_\xi}-\frac{\mu_\xi}{1+\lambda_\xi}u_\xi(t).
\end{split}
\end{equation}
Putting \eqref{SOL DL:ODE} into \eqref{EQ DL:Duh-2}, we get
\begin{equation}
\label{EQ DL:Duht-2}
\begin{split}
   \mathcal{D}_t^\alpha u_\xi(t) = & - \frac{\mu_\xi}{1+\lambda_\xi}\varphi_\xi E_{\alpha,1}\left(-\frac{\mu_\xi}{1+\lambda_\xi}t^\alpha\right)
   \\
   & + \frac{1}{1+\lambda_\xi} E_{\alpha,1}\left(-\frac{\mu_\xi}{1+\lambda_\xi}t^\alpha\right) f_\xi(0)
   \\
   & + \frac{1}{1+\lambda_\xi} \int_0^t E_{\alpha,1}\left(-\frac{\mu_\xi}{1+\lambda_\xi}s^\alpha\right) f'_{\xi}(t-s)ds.
\end{split}
\end{equation}
Substituting \eqref{EQ DL:Duht-2} into \eqref{EQ:D expansion Du-2}, we obtain
\begin{equation}
\label{EQ DL:Du-2}
\begin{split}
\mathcal{D}_t^\alpha u(t) = &
- \sum_{\xi\in\mathcal{I}}\frac{\mu_\xi}{1+\lambda_\xi}\varphi_\xi E_{\alpha,1}\left(-\frac{\mu_\xi}{1+\lambda_\xi}t^\alpha\right)e_\xi
\\
& + \sum_{\xi\in\mathcal{I}}\frac{1}{1+\lambda_\xi} E_{\alpha,1}\left(-\frac{\mu_\xi}{1+\lambda_\xi}t^\alpha\right) f_\xi(0) e_\xi
\\
& + \sum_{\xi\in\mathcal{I}}\left[\frac{1}{1+\lambda_\xi} \int_0^t E_{\alpha,1}\left(-\frac{\mu_\xi}{1+\lambda_\xi}s^\alpha\right) f'_{\xi}(t-s)ds\right]e_\xi.
\end{split}
\end{equation}
Applying the operator $\mathcal{L}$ to \eqref{EQ DL:Du-2} and taking into account formulas \eqref{rem}, we have
\begin{equation}
\label{EQ DL:DLu-2}
\begin{split}
\mathcal{D}_t^\alpha\mathcal{L} u(t)&=-\sum_{\xi\in\mathcal{I}}\frac{\mu_\xi}{1+\lambda_\xi}(\mathcal{L}\varphi, e_\xi)_\mathcal{H} E_{\alpha,1}\left(-\frac{\mu_\xi}{1+\lambda_\xi}t^\alpha\right)e_\xi
\\
&+\sum_{\xi\in\mathcal{I}}\frac{\lambda_\xi}{1+\lambda_\xi} E_{\alpha,1}\left(-\frac{\mu_\xi}{1+\lambda_\xi}t^\alpha\right) f_\xi(0)  e_\xi
\\
&+\sum_{\xi\in\mathcal{I}}\left[\frac{\lambda_\xi}{1+\lambda_\xi}
\int_0^t E_{\alpha,1}\left(-\frac{\mu_\xi}{1+\lambda_\xi}s^\alpha\right) f'_{\xi}(t-s)ds\right]e_\xi.
\end{split}
\end{equation}

Now, let us estimate $\mathcal H$-norms
\begin{equation}\label{C1: Sol1-mod-2}
\begin{split}
\|u(t)\|_{\mathcal{H}}^{2} &\leq \sum_{\xi\in\mathcal{I}}\left|E_{\alpha,1}\left(-\frac{\mu_\xi}{1+\lambda_\xi}t^\alpha\right) \right|^{2} |\varphi_\xi|^{2}
\\
&+\sum_{\xi\in\mathcal{I}}\left|\frac{f_\xi(t)}{\mu_\xi}\right|^{2} +\sum_{\xi\in\mathcal{I}}\left|\frac{f_\xi(0)}{\mu_\xi}\right|^{2}\left|E_{\alpha,1}\left(-\frac{\mu_\xi}{1+\lambda_\xi}t^\alpha\right) \right|^{2}
\\
&+\sum_{\xi\in\mathcal{I}}\left|\int_0^t \frac{f'_{\xi}(t-s)}{\mu_\xi} E_{\alpha,1}\left(-\frac{\mu_\xi}{1+\lambda_\xi}s^\alpha\right) ds\right|^{2}
\\
&
\leq C \sum\limits_{\xi\in\mathcal{I}} \left(\frac{1}{1+\frac{\mu_\xi}{1+\lambda_\xi}t^\alpha}\right)^{2} |\varphi_{\xi}|^{2}
\\
&+\sum_{\xi\in\mathcal{I}}\left|\frac{f_\xi(t)}{\mu_\xi}\right|^{2} +C \sum\limits_{\xi\in\mathcal{I}} \left(\frac{1}{1+\frac{\mu_\xi}{1+\lambda_\xi}t^\alpha}\right)^{2}\left|\frac{f_\xi(0)}{\mu_\xi}\right|^{2}
\\
& + C \sum\limits_{\xi\in\mathcal{I}}\frac{1}{\mu_\xi^2}\left(\int\limits_0^t \frac{1}{1+\frac{\mu_\xi}{1+\lambda_\xi}s^\alpha} |f'_{\xi}(t-s)| d s\right)^{2}
\end{split}
\end{equation}

\begin{equation*}
\begin{split}
&\leq C \sum\limits_{\xi\in\mathcal{I}} |\varphi_{\xi}|^{2}
+C\sum_{\xi\in\mathcal{I}}\left|\frac{f_\xi(t)}{\mu_\xi}\right|^{2}
\\
& + C \sum\limits_{\xi\in\mathcal{I}}\frac{1}{\mu_\xi^2}\int\limits_0^T \left(\frac{1}{1+\frac{\mu_\xi}{1+\lambda_\xi}s^\alpha}\right)^{2}  d s \,
\int\limits_0^T |f'_{\xi}(s)|^{2} d s
\\
&\leq C\|\varphi\|_\mathcal{H}^2 + C\|f(t)\|_{\mathcal{H}^{-1}_{\mathcal M}}^2 +  C \|f\|^{2}_{W^{1}(0, T; \mathcal H^{-1}_{\mathcal M})}.
\end{split}
\end{equation*}
Finally, we obtain
$$
\|u\|_{L^{2}(0, T; \mathcal H)}^{2}\leq C(\|\varphi\|_\mathcal{H}^2 +
  \|f\|^{2}_{W^{1}(0, T; \mathcal H^{-1}_{\mathcal M})}).
$$

By using \eqref{EST: Mittag} and arguing as in \eqref{C1: Sol1-mod-2}, from \eqref{FIN SOL DL:Pseudo-2}--\eqref{EQ DL:DLu-2} we get the following estimates
\begin{equation*}
\begin{split}
    \|\mathcal{L}u\|_{L^{2}([0,T],\mathcal{H})}^2&\leq C\|\mathcal{L}\varphi\|_\mathcal{H}^2+C\|f\|_{W^{1}([0,T],\mathcal{H}_{\mathcal L, \mathcal M}^{1, -1})}^2,
\end{split}
\end{equation*}

\begin{equation*}
\begin{split}
    \|\mathcal{M}u\|_{L^{2}([0,T],\mathcal{H})}^2&\leq C\|\mathcal{M}\varphi\|_\mathcal{H}^2+C\|f\|_{W^{1}([0,T],\mathcal{H})}^2,
\end{split}
\end{equation*}

\begin{equation*}
\begin{split}
  \|\mathcal{D}_t^\alpha u\|_{L^{2}([0,T],\mathcal{H})}^2 \leq C(\|\varphi\|_{\mathcal{H}^{-1, 1}_{\mathcal L, \mathcal M}}^2+\|f\|_{W^{1}([0,T],\mathcal{H}^{-1}_{\mathcal L})}^2),
\end{split}
\end{equation*}
and
\begin{equation*}
\begin{split}
\|\mathcal{D}_t^\alpha\mathcal{L} u\|_{L^{2}([0,T],\mathcal{H})}^2
\leq C(\|\varphi\|_{\mathcal{H}^{1}_{\mathcal M}}^2+\|f\|_{W^{1}([0,T],\mathcal{H})}^2),
\end{split}
\end{equation*}
respectively. It proves the existence result.

The proof of the uniqueness of the solution of Theorem \ref{Theorem D-L-2} is the same as in the case of Theorem \ref{Theorem D-L}.
\end{proof}

\section{Inverse problem}

This section is concerned with an inverse problem for
the pseudo-parabolic equation \eqref{EQ:PP}. We obtain existence and uniqueness results for this problem, by using the $\mathcal{L}$--Fourier method.
\begin{problem}
Find a pair of functions $(u(t), f)$ satisfying the inverse problem \eqref{EQ:PP}-\eqref{C:FIN}.
\end{problem}

Let us define $\gamma:=\max\{0, \kappa - 1\},$ where $\kappa$ is from Assumption \ref{A2}. A generalised solution of the inverse problem \eqref{EQ:PP}--\eqref{C:FIN} is the pair of functions $(u(t), f)$, where
$u\in C^{\alpha}([0,T];\mathcal H^{1+\gamma}_{\mathcal L} \bigcap \mathcal H^{1+\gamma}_{\mathcal M})$, and $f\in\mathcal H$.

For Problem \eqref{EQ:PP}--\eqref{C:FIN} the following statement holds true.

\begin{theorem} \label{Th}
Suppose that Assumptions \ref{A1} and \ref{A2} hold. Let $\varphi,\psi\in\mathcal H^{1+\gamma}_{\mathcal L} \bigcap \mathcal H^{1+\gamma}_{\mathcal M}$. Then the generalised solution of the inverse problem \eqref{EQ:PP}--\eqref{C:FIN} exists,
is unique, and can be written in the form
\begin{equation*}\begin{split}
u(t)=\varphi
&+\sum_{\xi\in\mathcal{I}}\frac{[(\psi, e_\xi)_\mathcal{H}-(\varphi, e_\xi)_\mathcal{H}] \, \left(1-E_{\alpha,1}(-\frac{\mu_\xi}{1+\lambda_\xi} t^\alpha)\right) \, e_\xi}{\left(1-E_{\alpha,1}(-\frac{\mu_\xi}{1+\lambda_\xi}T^\alpha)\right)},
\end{split}\end{equation*}
\begin{equation*}\begin{split}
f=\mathcal{M}\varphi
&+\sum_{\xi\in\mathcal{I}}\frac{[(\mathcal{M}\psi,e_\xi)_\mathcal{H}-(\mathcal{M}\varphi,e_\xi)_\mathcal{H}]\, e_\xi}
{1-E_{\alpha,1}(-\frac{\mu_\xi}{1+\lambda_\xi}T^\alpha)}.
\end{split}\end{equation*}
\end{theorem}

\begin{proof} {\it Existence.} Since the system $\{e_\xi\}_{\xi\in\mathcal I}$ is a basis in the space $\mathcal{H}$, we expand the functions $u(t)$ and $f$ as follows:
\begin{equation}\label{EQ: expansion u}
u(t)=\sum_{\xi\in\mathcal{I}}u_\xi(t)e_\xi,
\end{equation}
and
\begin{equation}\label{EQ: expansion f}
f=\sum_{\xi\in\mathcal{I}}f_\xi e_\xi,
\end{equation}
where $u_\xi(t)$ and $f_\xi$ are
\begin{equation*}
\begin{split}
    u_\xi(t)&=(u(t),e_\xi)_\mathcal{H},\;\xi\in\mathcal{I},\\
    f_\xi&=(f,e_\xi)_\mathcal{H},\;\xi\in\mathcal{I}.
\end{split}
\end{equation*}

Substituting \eqref{EQ: expansion u} and \eqref{EQ: expansion f} into the equations \eqref{EQ:PP}--\eqref{C:FIN} and using the relations
$$
\mathcal{L}e_\xi=\lambda_\xi e_\xi, \,\,\, \mathcal{M}e_\xi=\mu_\xi e_\xi,
$$
we get the following problem for the functions $u_\xi(t)$ and for the constants $f_\xi,\,\,\xi\in\mathcal{I}$:
\begin{equation}\label{EQ IP:ODE}
\mathcal D_t^\alpha u_\xi(t)+\frac{\mu_\xi}{1+\lambda_\xi} u_\xi(t)=\frac{f_\xi}{1+\lambda_\xi},
\end{equation}
\begin{equation}\label{CON IP:ODE IN}
   u_\xi(0)=\varphi_\xi,
\end{equation}
\begin{equation}\label{CON IP:ODE FIN}
u_\xi(T)=\psi_\xi,
\end{equation}
for $t\in[0,T]$ and for any $\xi\in\mathcal{I}$. Where $\varphi_\xi,\;\psi_\xi$ are the coefficients of the expansions of $\varphi,\;\psi$, i.e.
\begin{eqnarray}\label{EQ:D expansion phi+psi}
\varphi=\sum_{\xi\in\mathcal{I}}\varphi_\xi e_\xi, \,\,\,
\psi=\sum_{\xi\in\mathcal{I}}\psi_\xi e_\xi,
\end{eqnarray}
given by
\begin{equation}\label{coef phi+psi}
\varphi_\xi=(\varphi, e_\xi)_\mathcal{H}, \,\,\, \psi_\xi=(\psi, e_\xi)_\mathcal{H}.
\end{equation}

We seek a general solution of Problem \eqref{EQ IP:ODE}--\eqref{CON IP:ODE FIN} in the following form
\begin{equation}\label{General}
u_\xi(t)=\frac{f_\xi}{\mu_\xi}+C_\xi E_{\alpha,1}(-\frac{\mu_\xi}{1+\lambda_\xi} t^\alpha),
\end{equation}
where the constants $C_\xi, f_\xi$ are unknown. By using the conditions \eqref{C:IN}--\eqref{C:FIN} we can find them.

We first find $C_\xi$:
$$
u_\xi(0)=\frac{f_\xi}{\mu_\xi}+C_\xi=\varphi_\xi,
$$
$$
u_\xi(T)=\frac{f_\xi}{\mu_\xi}+C_\xi E_{\alpha,1}(-\frac{\mu_\xi}{1+\lambda_\xi}T^\alpha)=\psi_\xi,
$$
$$
\varphi_\xi-C_\xi+C_\xi E_{\alpha,1}(-\frac{\mu_\xi}{1+\lambda_\xi} T^\alpha)=\psi_\xi.
$$
Then
$$
C_\xi=\frac{\varphi_\xi-\psi_\xi}{1-E_{\alpha,1}(-\frac{\mu_\xi}{1+\lambda_\xi} T^\alpha)}.
$$
Then $f_\xi$ is represented as
$$
f_\xi=\mu_\xi \varphi_\xi-\mu_\xi C_\xi.
$$
Substituting $f_\xi,\,u_\xi(t)$ into the expansions \eqref{EQ: expansion u} and \eqref{EQ: expansion f}, we find
$$
u(t)=\varphi+\sum_{\xi\in\mathcal{I}}C_\xi\left(E_{\alpha,1}(-\frac{\mu_\xi}{1+\lambda_\xi}t^\alpha)-1\right)e_\xi,
$$
$$
f=\sum_{\xi\in\mathcal{I}}\mu_\xi\varphi_\xi e_\xi-\sum_{\xi\in\mathcal{I}}\mu_\xi C_\xi e_\xi.
$$
By using the self-adjointness of the operator $\mathcal{M}$,
$$
(\mathcal{M}\varphi,e_\xi)_\mathcal{H}=(\varphi,\mathcal{M}e_\xi)_\mathcal{H},
$$
and using $\mathcal{M}e_\xi=\mu_\xi e_\xi$, we obtain
$$
(\varphi, e_\xi)_\mathcal{H}=\frac{(\mathcal{M}\varphi, e_\xi)_\mathcal{H}}{\mu_\xi}, \,\,\, (\psi, e_\xi)_\mathcal{H}=\frac{(\mathcal{M}\psi, e_\xi)_\mathcal{H}}{\mu_\xi}.
$$
Substituting these identities into the formula of $C_\xi$, we get that
$$
C_\xi=\frac{(\mathcal M\varphi, e_\xi)_\mathcal{H}-(\mathcal M\psi, e_\xi)_\mathcal{H}}{\mu_\xi\left(1-E_{\alpha,1}(-\frac{\mu_\xi}{1+\lambda_\xi} T^\alpha)\right)}.
$$
Then, formally, one obtains
\begin{equation}
\label{EQ: U Solution pseudo}
u(t)=\varphi+\sum_{\xi\in\mathcal{I}}
\frac{[(\psi,e_\xi)_\mathcal{H}-(\varphi,e_\xi)_\mathcal{H}]\, \left(1-E_{\alpha,1} (-\frac{\mu_\xi}{1+\lambda_\xi} t^\alpha)\right)\, e_\xi}{\left(1-E_{\alpha,1} (-\frac{\mu_\xi}{1+\lambda_\xi} T^\alpha)\right)},
\end{equation}
\begin{equation}
\label{EQ: f Solution pseudo}
f=\mathcal{M}\varphi+\sum_{\xi\in\mathcal{I}}
\frac{[(\mathcal{M}\psi,e_\xi)_\mathcal{H}-(\mathcal{M}\varphi,e_\xi)_\mathcal{H}]\,  e_\xi}{1 - E_{\alpha,1} (-\frac{\mu_\xi}{1+\lambda_\xi} T^\alpha)}.
\end{equation}

Since $T>T_0\geq 0,\;T_0=const$, for denominators of \eqref{EQ: U Solution pseudo} and \eqref{EQ: f Solution pseudo}, the following estimate holds true by \eqref{EST: Mittag AB},
\begin{equation}
\label{EST: Mittag B}
\begin{split}
1-E_{\alpha,1} (-\frac{\mu_\xi}{1+\lambda_\xi} T^\alpha)&\geq  1-\frac{1}{1+\frac{\mu_\xi}{1+\lambda_\xi} T^\alpha\Gamma(1+\alpha)^{-1}}
\\
&=\frac{\frac{\mu_\xi}{1+\lambda_\xi} T^\alpha\Gamma(1+\alpha)^{-1}}{1+\frac{\mu_\xi}{1+\lambda_\xi} T^\alpha\Gamma(1+\alpha)^{-1}}\\& =
\frac{\Gamma(1+\alpha)^{-1}}{\frac{1+\lambda_\xi}{\mu_\xi T^\alpha}+\Gamma(1+\alpha)^{-1}}\\&\geq M>0.
\end{split}
\end{equation}

Here, by Assumption \ref{A2} we have $|\lambda_\xi|=O(|\mu_\xi|^\kappa)$ as $|\xi|\to\infty$ for some $\kappa>0$. In the case if $\kappa\leq 1$ the estimate \eqref{EST: Mittag B} makes a sense. Now, suppose that $\kappa > 1$. Then, we have
\begin{equation}
\label{EST: Mittag B2}
\begin{split}
   |\mu_\xi|^{\kappa-1}(1-E_{\alpha,1} (-\frac{\mu_\xi}{1+\lambda_\xi} T^\alpha))\geq |\mu_\xi|^{\kappa-1}\frac{\Gamma(1+\alpha)^{-1}}{\frac{1+\lambda_\xi}{\mu_\xi T^\alpha}+\Gamma(1+\alpha)^{-1}}\geq M>0.
\end{split}
\end{equation}

According to \cite{LG99}, we have
\begin{equation}\label{FORMULA:Derivative E}
\mathcal{D}_t^\alpha\left(E_{\alpha,1}(-\lambda t^\alpha)\right)=-\lambda E_{\alpha,1}(-\lambda t^\alpha).
\end{equation}
Now, we prove the convergence of the obtained
infinite series corresponding to the functions $u(t),\;\mathcal{D}_t^\alpha u(t),\;\mathcal{M}u(t),\;\mathcal{D}_t^\alpha \mathcal{L}u(t)$, and $f$.

Before we get the convergence, let us calculate $\mathcal{D}_t^\alpha u(t),\;\mathcal{M}u(t)$ and $\mathcal{D}_t^\alpha \mathcal{L}u(t)$.
Applying the operator $\mathcal{D}_t^\alpha$ to \eqref{EQ: U Solution pseudo}, and using \eqref{FORMULA:Derivative E}, we have
\begin{equation}\label{DERIV: by t}
\begin{split}
\mathcal{D}_t^\alpha u(t)&=\sum_{\xi\in\mathcal{I}}
\frac{[(\mathcal{M}\psi,e_\xi)_\mathcal{H}-(\mathcal{M}\varphi,e_\xi)_\mathcal{H}]\,\mathcal{D}_t^\alpha\left(1-E_{\alpha,1} (-\frac{\mu_\xi}{1+\lambda_\xi} t^\alpha)\right)\, e_\xi}{\mu_\xi\left(1-E_{\alpha,1} (-\frac{\mu_\xi}{1+\lambda_\xi} T^\alpha)\right)}\\
&=\sum_{\xi\in\mathcal{I}}
\frac{[(\mathcal{M}\psi, e_\xi)_\mathcal{H}-(\mathcal{M}\varphi, e_\xi)_\mathcal{H}]\, E_{\alpha,1} (-\frac{\mu_\xi}{1+\lambda_\xi} t^\alpha)\, e_\xi}{(1+\lambda_\xi)\,\left(1-E_{\alpha,1} (-\frac{\mu_\xi}{1+\lambda_\xi} T^\alpha)\right)}.
\end{split}
\end{equation}
Applying the operators $\mathcal{L}$ and $\mathcal{M}$ to \eqref{EQ: U Solution pseudo} and taking into account \eqref{rem}, we have
$$
\mathcal{L}u(t)=\mathcal{L}\varphi+\sum_{\xi\in\mathcal{I}}
\frac{[(\mathcal{L}\psi,e_\xi)_\mathcal{H}-(\mathcal{L}\varphi,e_\xi)_\mathcal{H}]\,\left(1-E_{\alpha,1} (-\frac{\mu_\xi}{1+\lambda_\xi} t^\alpha)\right)\, e_\xi}{1-E_{\alpha,1} (-\frac{\mu_\xi}{1+\lambda_\xi} T^\alpha)},
$$
\begin{equation}\label{DERIV: by x}
\begin{split}
\mathcal{M}u(t)=\mathcal{M}\varphi+\sum_{\xi\in\mathcal{I}}
\frac{[(\mathcal{M}\psi,e_\xi)_\mathcal{H}-(\mathcal{M}\varphi,e_\xi)_\mathcal{H}]\,\left(1-E_{\alpha,1} (-\frac{\mu_\xi}{1+\lambda_\xi} t^\alpha)\right)\, e_\xi}{1-E_{\alpha,1} (-\frac{\mu_\xi}{1+\lambda_\xi} T^\alpha)},
\end{split}
\end{equation}
respectively.

Now, applying the operator $\mathcal{D}_t^\alpha$ to the first equality in \eqref{DERIV: by x}, and taking into account formulas \eqref{rem}, we have
\begin{equation}\label{DERIV: by xt}
\begin{split}
\mathcal{D}_t^\alpha\mathcal{L}u(t)&=\sum_{\xi\in\mathcal{I}}
\frac{\mu_\xi[(\mathcal{L}\psi, e_\xi)_\mathcal{H}-(\mathcal{L}\varphi, e_\xi)_\mathcal{H}] \, E_{\alpha,1} (-\frac{\mu_\xi}{1+\lambda_\xi} t^\alpha)\, e_\xi}{(1+\lambda_\xi) \, \left(1-E_{\alpha,1} (-\frac{\mu_\xi}{1+\lambda_\xi} T^\alpha)\right)}.
\end{split}
\end{equation}

Let us recall that $\gamma =\max\{0, \kappa - 1\}.$
By using the formulas \eqref{EQ: U Solution pseudo}--\eqref{DERIV: by xt}, and taking into account estimates \eqref{EST: Mittag}, we get the following estimates
\begin{equation}
\label{first}
\|u\|_{C([0,T],\mathcal{H})}^2 \leq C(\|\varphi\|_\mathcal{H}^2 + \|\varphi\|_{\mathcal H_{\mathcal M}^{\gamma}}^2 + \|\psi\|_{\mathcal H_{\mathcal M}^{\gamma}}^2),
\end{equation}
$$
\|\mathcal{M} u\|_{C([0,T],\mathcal{H})}^2 \leq C(\|\varphi\|_{\mathcal H_{\mathcal M}^{1}}^2 + \|\varphi\|_{\mathcal H_{\mathcal M}^{1+\gamma}}^2 + \|\psi\|_{\mathcal H_{\mathcal M}^{1+\gamma}}^2),
$$
$$
\|\mathcal D_t^\alpha u\|_{C([0,T],\mathcal{H})}^2 \leq C(\|\varphi\|_{\mathcal H_{\mathcal L, \mathcal M}^{-1, 1+\gamma}}^2 + \|\psi\|_{\mathcal H_{\mathcal L, \mathcal M}^{-1, 1+\gamma}}^2),
$$
$$
\|\mathcal D_t^\alpha\mathcal{L}u\|_{C([0,T],\mathcal{H})}^2 \leq C(\|\varphi\|_{\mathcal H_{\mathcal M}^{1+\gamma}}^2 + \|\psi\|_{\mathcal H_{\mathcal M}^{1+\gamma}}^2).
$$
For clarity, we only show the first estimate. By taking the $\mathcal{H}$-norm from the both sides of the representation \eqref{EQ: U Solution pseudo}, we obtain
\begin{equation}
\label{EQ: U Solution pseudo_H-est}
\begin{split}
\|u(t)\|_{\mathcal H}^{2}\leq\|\varphi\|_{\mathcal H}^{2}&+\sum_{\xi\in\mathcal{I}}\left[|\varphi_\xi|^{2}+|\psi_\xi|^{2}\right]
\left|\frac{\left(1-E_{\alpha,1} (-\frac{\mu_\xi}{1+\lambda_\xi} t^\alpha)\right)}{\left(1-E_{\alpha,1} (-\frac{\mu_\xi}{1+\lambda_\xi} T^\alpha)\right)}\right|^{2}.
\end{split}
\end{equation}
Now, by using the estimates \eqref{EST: Mittag}, \eqref{EST: Mittag B} and \eqref{EST: Mittag B2}, we get
\begin{equation}
\label{EQ: U Solution pseudo_H-est_2}
\begin{split}
\|u(t)\|_{\mathcal H}^{2}&\leq\|\varphi\|_{\mathcal H}^{2}+C_1\sum_{\xi\in\mathcal{I}}|\mu_\xi|^{2\gamma}\left[|\varphi_\xi|^{2}+|\psi_\xi|^{2}\right]
\\
&\leq C(\|\varphi\|_\mathcal{H}^2 + \|\varphi\|_{\mathcal H_{\mathcal M}^{\gamma}}^2 + \|\psi\|_{\mathcal H_{\mathcal M}^{\gamma}}^2),
\end{split}
\end{equation}
for some constants $C_1>0$ and $C>0$. Thus, we finish the proof of \eqref{first}.

Similarly, for the source term $f$, one obtains the estimate
\begin{eqnarray*}
\|f\|_\mathcal{H}^2\leq C(\|\varphi\|_{\mathcal H_{\mathcal M}^{1}}^2 + \|\varphi\|_{\mathcal H_{\mathcal M}^{1+\gamma}}^2 + \|\psi\|_{\mathcal H_{\mathcal M}^{1+\gamma}}^2).
\end{eqnarray*}

Existence of the solution of Problem \eqref{EQ:PP}-\eqref{C:FIN} is proved.

{\it Proof of the uniqueness result.} Let us suppose that $\{u_1(t),f_1\}$ and $\{u_2(t),f_2\}$ are solution of the Problem \eqref{EQ:PP}-\eqref{C:FIN}. Let $u(t)=u_1(t)-u_2(t)$ and
$f=f_1-f_2$. Then $u(t)$ and $f$ satisfy
\begin{equation}\label{EQ: pseudo with hom con}
\mathcal D_t^\alpha[u(t)+\mathcal{L}u(t)]+\mathcal{M}u(t)=f,
\end{equation}

\begin{equation}\label{EQ: IN-CON-Pseudo 0}
u(0)=0,
\end{equation}

\begin{equation}\label{EQ: END-CON-Pseudo 0}
u(T)=0.
\end{equation}
We also have
\begin{equation}\label{PHE-1}
u_\xi(t)=(u(t),e_\xi)_{\mathcal H},\,\,\, \xi\in\mathcal I,
\end{equation}
and
\begin{equation}\label{PHE-2}
f_\xi=(f, e_\xi)_{\mathcal H},\,\,\, \xi\in\mathcal I.
\end{equation}
Applying the operator $\mathcal D_t^\alpha$ to \eqref{PHE-1}, we have
$$
\mathcal (1+\lambda_\xi)\mathcal D_t^\alpha u_\xi(t)=(\mathcal D_t^\alpha [u(t)+\mathcal{L} u(t)],e_\xi)_{\mathcal H}=(-\mathcal{M}u(t)+f, e_\xi)_{\mathcal H}=-\mu_\xi u_\xi(t)+f_\xi.
$$

Thus, we get the problem with homogeneous conditions. The general solution of this equation has the form \eqref{General}.
Using the homogeneous conditions $u_\xi(0)=0$ and $u_\xi(T)=0$ we obtain
$$
f_\xi=0\,\,\,\textrm{and}\,\,\,  u_\xi(t)\equiv0.
$$
Further, by the completeness of the system $\{e_\xi\}_{\xi\in\mathcal I}$ in $\mathcal H$, we obtain $f\equiv0, u(t)\equiv0.$
Uniqueness of the solution of Problem \eqref{EQ:PP}-\eqref{C:FIN} is proved.
\end{proof}

\end{document}